\documentclass[11pt,epsf]{article}
\usepackage{graphicx}
\usepackage{amsthm}
\usepackage{amsfonts}
\usepackage{amssymb}
\title{Addendum: A separation in modulus property of the 
zeros of a partial theta function}
\author{Vladimir Petrov Kostov\\ 
Universit\'e C\^ote d'Azur, CNRS, LJAD, France,\\  
e-mail: kostov@math.unice.fr} 
\date{}
\newtheorem{tm}{Theorem}

\newtheorem{defi}[tm]{Definition}
\newtheorem{rem}[tm]{Remark}
\newtheorem{rems}[tm]{Remarks}
\newtheorem{lm}[tm]{Lemma}

\newtheorem{nota}[tm]{Notation}
\begin{document} 
\maketitle 
\begin{abstract}
We consider the partial theta function 
$\theta (q,z):=\sum _{j=0}^{\infty}q^{j(j+1)/2}z^j$, 
where $z\in \mathbb{C}$ is a variable and $q\in \mathbb{C}$, $0<|q|<1$,  
is a parameter. Set $D(a):=\{ q\in \mathbb{C}$, $0<|q|\leq a$, 
$\arg (q)\in [\pi /2,3\pi /2]\}$. We show that for $k\in \mathbb{N}$ and 
$q\in D(0.55)$, 
there exists exactly one zero of $\theta (q,.)$ (which is a simple one) 
in the open annulus $|q|^{-k+1/2}<z<|q|^{-k-1/2}$ (if $k\geq 2$) or in the 
punctured disk $0<z<|q|^{-3/2}$ (if $k=1$). For $k=1$, $4$, $5$, $6$, $\ldots$, 
this holds true for $q\in D(0.6)$ as well. \\


\end{abstract}

\section{The new results}

In this addendum to~\cite{Ko8}  
we consider the {\em partial theta function} 
$\theta (q,z):=\sum _{j=0}^{\infty}q^{j(j+1)/2}z^j$. The series converges for 
$q\in \mathbb{D}_1$, $z\in \mathbb{C}$, where $\mathbb{D}_a$ 
denotes the open disk centered at $0$ and of radius $a$. We regard $q$ as a 
parameter and $z$ as a variable. 

\begin{defi}
{\rm For $0<a<b$ and for $q\in \mathbb{D}_1\setminus 0$ fixed, 
denote by $U_{a,b}\subset \mathbb{C}$ the open annulus 
$\{ |q|^{-a}<|z|<|q|^{-b}\}$. For $q$ fixed, we say that  
{\em strong separation (in modulus)} of the zeros of $\theta$ 
takes place for $k\geq k_0$, 
if for any $k\geq k_0$, there exists a unique zero $\xi _k$ 
of $\theta (q,.)$ which is 
a simple one and}
\begin{equation}\label{eqtwosides}
|\xi _k|\in U_{k-1/2,k+1/2}~~~~\,,~~~\, {\rm if}~~~\, k\geq 2~~~\,,~~~\, 
{\rm and}~~~\, 
\xi _k\in (\mathbb{D}_{|q|^{-3/2}}\setminus 0)~~~~\, ,~~~\, {\rm if}~~~\, k=1~, 
\end{equation}
{\rm while the remaining $k_0-1$ zeros (counted with multiplicity) are 
in $\mathbb{D}_{|q|^{-k_0+1/2}}\setminus 0$. If $k_0=1$, then we say that 
the zeros of $\theta$ 
{\em are strongly separated}. (One can notice that the cases $k_0=1$ and 
$k_0=2$ are identical.)}
\end{defi}

It is shown in \cite{Ko8} that the zeros of $\theta$ are strongly separated  
for $0<|q|\leq c_0:=0.2078750206\ldots$, 
see Lemma~1 and its proof therein. 
Set $D(a):=\{ q\in \mathbb{C}$, $0<|q|\leq a$, 
$\arg (q)\in [\pi /2,3\pi /2]\}$.

\begin{tm}\label{maintmadd}
(1) For $q\in D(0.55)$, the zeros of $\theta$ are strongly separated. 

(2) For $q\in D(0.6)$, conditions (\ref{eqtwosides}) hold true for 
$k=1$, $4$, $5$, $6$, $\ldots$, and there are exactly two zeros of $\theta$ 
(counted with multiplicity) in $U_{3/2,7/2}$.  
\end{tm}

\begin{rems}
{\rm (1) The domain $D(0.55)\cup (\mathbb{D}_{c_0}\setminus 0)$ 
is much larger than the domain 
$\mathbb{D}_{c_0}\setminus 0$ in which strong separation of the zeros of 
$\theta$ is guaranteed by Lemma~1 in~\cite{Ko8}. It is impossible  
to extend Theorem~\ref{maintmadd} to the whole of 
$\overline{\mathbb{D}_{0.55}\setminus 0}$, 
because for certain values of $q$ in the right half of 
$\mathbb{D}_{0.55}\setminus 0$, 
the function $\theta (q,.)$ 
has double zeros, see~\cite{KoSh} and~\cite{Ko8}. 

(2) Numerical computations suggest that in part (1) 
of Theorem~\ref{maintmadd} the number $0.55$ can be replaced by $0.6$, see 
Remark~\ref{remnum}.}
\end{rems}

\section{Comments}

Theorem~\ref{maintmadd} can be compared with the results of \cite{Ko8}. 
Set $\alpha _0:=\sqrt{3}/2\pi =0.2756644477\ldots$. 
Parts (1) and (4) of Theorem~5 in \cite{Ko8} read:

\begin{tm}\label{tmold}
(1) For $n\geq 5$ and for $|q|\leq 1-1/(\alpha _0n)$, strong separation of the 
zeros $\xi _k$ of $\theta$ occurs for $k\geq n$.

(2) For $0<|q|\leq 1/2$, strong separation of the 
zeros of $\theta$ occurs for $k\geq 4$.
\end{tm}

Part (1) of the theorem implies that for $k\geq n\geq 5$, one has 

$$(1-1/(\alpha _0k))^{-k+1/2}~=:~m_k~\leq ~|\xi _k|~\leq ~M_k~:=~
(1-1/(\alpha _0k))^{-k-1/2}~.$$
The following table gives an idea how the numbers $m_n$ and $M_n$ decrease as 
$n$ increases (both of them tend to $e^{1/\alpha _0}=37.62236657\ldots$ as 
$n\rightarrow \infty$). We list the truncations up to the second decimal 
for the numbers $\tau _n:=1-1/(\alpha _0n)$ and up to the first decimal 
for $m_n$ and $M_n$. The numbers $m_n$ and $M_n$ are the internal 
and external radius of the open annulus to which the zero $\xi _n$ belongs for 
$|q|=\tau _n$.

$$\begin{array}{ccrrrrrrrrrr}
n&&5&6&7&8&9&10&15&20&25&30\\ \\ 
\tau _n&&0.27&0.39&0.48&0.54&0.59&0.63&0.75&0.81&0.85&0.87\\ \\ 
m_n&&336.2&164.5&115.2&92.8&80.2&72.2&55.3&49.5&46.5&44.7\\ \\ 
M_n&&1225.1&416.1&239.1&169.8&134.4&113.4&73.0&60.5&54.4&50.9\\ \\ 
\end{array}$$

\section{Proof of Theorem~\protect\ref{maintmadd}}

We use the same method of proof as the one of the proof of 
Theorem~5 in~\cite{Ko8}, but 
with more accurate estimations. We remind that (see~\cite{Ko8}) 
$\theta =\Theta ^*-G$, where 
$\Theta ^*(q,z):=\sum _{j=-\infty}^{\infty}q^{j(j+1)/2}z^j$ and 
$G(q,z):=\sum _{-\infty}^{-1}q^{j(j+1)/2}z^j=\sum _{j=1}^{\infty}q^{j(j-1)/2}z^{-j}$. Set 

$$Q~:=~\prod _{j=1}^{\infty}(1-q^j)~~~,~~~
U~:=~\prod_{j=1}^{\infty}(1+zq^j)~~~{\rm and}~~~
R~:=~\prod_{j=1}^{\infty}(1+q^{j-1}/z)~.$$
We recall that by formula (4) in \cite{Ko8} (resulting from the 
Jacobi triple product), one has $\Theta ^*=QUR$.

\begin{lm}\label{lmQ}
For $q\in D(0.6)$, one has $|Q|\geq 1.2$.
\end{lm}

\begin{proof}
We set $Q:=Q_0Q_1$, where $Q_0:=\prod _{j=1}^{11}(1-q^j)$ and 
$Q_1:=\prod _{j=12}^{\infty}(1-q^j)$. For the quantity $|Q_1|$ one obtains 
the minoration

$$|Q_1|\geq \prod _{j=12}^{\infty}(1-|q|^j)\geq \prod _{j=12}^{\infty}(1-0.6^j)
=\gamma :=0.9945691384\ldots ~.$$
Hence to prove the lemma it suffices to show that for $q\in D(0.6)$, 
the inequality 

$$|Q_0|\geq 1.2/\gamma =1.206552620\ldots$$
holds true. This can be proved using the maximum principle. The polynomial 
$Q_0$ has no zeros in $D(0.6)$ hence 
$\min _{\overline{D(0.6)}}|Q_0|=1/\max _{\overline{D(0.6)}}|1/Q_0|$ 
is attained on the border $\partial D(0.6)$ of the domain $D(0.6)$. 
The restrictions of $|Q_0|^2$ to 
the segment (resp. to the arc) of $\partial D(0.6)$ are 
a polynomial and a trigonometric polynomial respectively 
and the latter displayed 
inequality is readily checked numerically. The check can be limited to the 
upper half-plane, because all coefficients of $Q$ are real hence $Q(\bar{q})=
\overline{Q(q)}$.   
\end{proof}

\begin{lm}\label{lmk5}
For $q\in D(0.6)$, $k\in \mathbb{N}$, $k\geq 5$ and 
$|z|=|q|^{-k+1/2}$, one has $\theta (q,z)\neq 0$.
\end{lm}

\begin{proof}
One can minorize $|U|$ and $|R|$ as follows:

$$\begin{array}{llllcll}
|U|&\geq&\prod_{j=1}^{\infty}||zq^j|-1|&=&
\prod_{j=1}^{\infty}||q|^{-k+j+1/2}-1|&\geq&\prod_{j=1}^{\infty}|0.6^{-k+j+1/2}-1|\\ \\ 
&\geq&
\prod_{j=1}^{\infty}|0.6^{-9/2+j}-1|&=&
\eta&:=&0.2411047426\ldots \end{array}$$
and 
$$\begin{array}{llllcll}
|R|&\geq&\prod_{j=0}^{\infty}|1-|q^j/z||&=&
\prod_{j=0}^{\infty}|1-|q|^{k+j-1/2}|
&\geq&\prod_{j=0}^{\infty}|1-|q|^{9/2+j}|\\ \\ &\geq&
\prod_{j=0}^{\infty}|1-0.6^{9/2+j}| 
&=&\xi&:=&0.7715882456\ldots ~.\end{array}$$
Thus using Lemma~\ref{lmQ} one obtains the inequality 

\begin{equation}\label{eqQRU}
|QRU|\geq 1.2\cdot \eta \cdot \xi =0.2232403024\ldots ~.
\end{equation}
On the other hand the quantity $|G|$ can be majorized as follows:

\begin{equation}\label{eqG}
\begin{array}{lllllll}
|G|&\leq&\sum _{j=1}^{\infty}|q|^{j(j-1)/2}/|z|^j&=&
\sum _{j=1}^{\infty}|q|^{j(j-1)/2+(k-1/2)j}&&\\ \\ &\leq& 
\sum _{j=1}^{\infty}|q|^{j(j-1)/2+9j/2}&\leq&
\sum _{j=1}^{\infty}0.6^{j(j-1)/2+9j/2}&=&0.1066576686\ldots ~.\end{array}
\end{equation}
The lemma follows from inequalities (\ref{eqQRU}) and (\ref{eqG}) -- for 
$|z|=|q|^{-k+1/2}$, one has at the same time $\theta =\Theta ^*-G$, 
$|\Theta ^*|>0.22$ and $|G|<0.11$, so $\theta =0$ is impossible.
\end{proof}

\begin{lm}\label{lmk4}
For $q\in D(0.6)$ and $|z|=|q|^{-7/2}$, one has $\theta (q,z)\neq 0$.
\end{lm}

\begin{lm}\label{lmk1}
For $q\in D(0.6)$ and $|z|=|q|^{-3/2}$, one has $\theta (q,z)\neq 0$.
\end{lm}

\begin{lm}\label{lmk2}
For $q\in D(0.55)$ and $|z|=|q|^{-5/2}$, one has $\theta (q,z)\neq 0$.
\end{lm}

The last three lemmas are proved in 
Sections~\ref{secprlmk4}, \ref{secprlmk1} and 
\ref {secprlmk2} respectively. We explain how Theorem~\ref{maintmadd} 
results from them and from Lemma~\ref{lmk5}.  
We recall that for 
$|q|\leq 0.108$, all zeros of $\theta (q,.)$ are simple (see~\cite{Ko4}). 
For any $k\in \mathbb{N}$ fixed and for $|q|$ sufficiently small, 
there exists a single zero of $\theta (q,.)$ which is $\sim -q^{-k}$ as 
$q\rightarrow 0$, see Proposition~10 in~\cite{Ko2}. Hence for $|q|$ 
sufficiently small, this zero satisfies the conditions (\ref{eqtwosides}). 
These conditions hold true as $q$ varies along any 
segment $S$ 
belonging to a half-line passing through the origin and such that 
$S\subset (D(0.55)\cup \mathbb{D}_{c_0})$. Hence these conditions hold true for  
$q\in D(0.55)$. In the same way one sees that, with the possible exception of 
$\xi _2$ and $\xi _3$, the zeros 
of $\theta$ satisfy conditions (\ref{eqtwosides}) for $q\in D(0.6)$.

\section{Proof of Lemma~\protect\ref{lmk4}
\protect\label{secprlmk4}}

For $|G|$ we obtain the majoration 

$$|G|\leq \sum _{j=1}^{\infty}|q|^{j(j-1)/2+7j/2}\leq 
\sum _{j=1}^{\infty}0.6^{j(j-1)/2+7j/2}=0.1851580824\ldots ~.$$ 
We minorize the quantity 
$|Q|$ using Lemma~\ref{lmQ}. There remains to 
minorize $|U|$ and $|R|$.

We observe that $\theta (\bar{q},\bar{z})=\overline{\theta (q,z)}$, therefore 
when $|\theta |$ is majorized, we can assume that $\arg (q)\in [0,\pi ]$. 
Suppose that $q\in D(0.6)$, $\arg (q)\in [3\pi /4,\pi ]$. We define the 
sectors $S_j$ in $\mathbb{C}$ by the fomula

$$S_j:=\{ \arg (z)\in [(j-1)\pi /4,j\pi /4)\}~,~~~\, 
j=1~,~~~2~,~~~3~~~\, ,~~~\,  
S_4:=\{ \arg (z)\in [3\pi /4,\pi ]\} ~,$$ 
and $S_{-j}$ as the symmetric of $S_j$ w.r.t. the real axis. Suppose that 
$\zeta \in S_{\pm j}$. Set $\psi :=\arg (\zeta )$. By the cosine theorem 

$$|1-\zeta |^2=1+|\zeta |^2-2|\zeta |\cos (\psi )~.$$
The function $\cos (t)$ being even and decreasing on $[0,\pi ]$ one obtains the 
minoration

\begin{equation}\label{eqminor}
|1-\zeta |\geq \mu _j(|\zeta |):=(1+|\zeta |^2-
2|\zeta |\cos ((j-1)\pi /4))^{1/2}~.
\end{equation}
\begin{rems}\label{remsmu}
{\rm (1) For $\zeta \neq 0$, one has 
$\mu _4(|\zeta |)>\mu _3(|\zeta |)>\mu _2(|\zeta |)>\mu _1(|\zeta |)$. 

(2) For $j$ fixed and for $|\zeta |\geq 1$, 
the right-hand side of (\ref{eqminor}) 
is an increasing function in $|\zeta |$.

(3) One can prove by straightforward computation that for 
$1\geq |\zeta _1|>|\zeta _2|$ and for $3\geq \ell >m\geq 1$, one has 
$\mu _{\ell}(|\zeta _1|)\cdot \mu _m(|\zeta _2|)>
\mu _{\ell}(|\zeta _2|)\cdot \mu _m(|\zeta _1|)$.}
\end{rems} 

Consider for $|z|=|q|^{-7/2}$ (i.e. for $z=|q|^{-7/2}\omega$, $|\omega |=1$), 
the moduli of three consecutive factors 
$u_k:=1+zq^k$ of $U$, for $k=k^*$, $k^*+1$ and $k^*+2$. (The role of $\zeta$ 
will be played 
by the numbers $-zq^k$.) Notice that the numbers $u_k$ are of the form 

\begin{equation}\label{eqform}
1+|q|^{-7/2+k}\omega _k~~~\, ,~~~\, |\omega _k|=1~.
\end{equation}
\begin{rems}\label{remsminoration}
{\rm (1) At least one of the three numbers $-zq^k$ 
belongs to the left half-plane, because $\arg (q)\in [\pi /2,\pi ]$. Hence to 
the corresponding modulus $|u_k|$ minoration $\mu _3$ is applicable. If 
at most one of the other two numbers $-zq^k$ belongs to $S_1\cup S_{-1}$, 
then to the corresponding modulus $|u_k|$ minoration $\mu _1$,  
and to the third modulus minoration $\mu _2$ are applicable respectively. 

(2) Suppose that at least two of the three numbers $-zq^k$ belong to 
$S_1\cup S_{-1}$. Then these correspond to $k=k^*$ and $k=k^*+2$, because 
$\arg (q)\in [\pi /2,\pi ]$. Moreover, $-zq^{k^*+1}\in S_4\cup S_{-4}$, 
so minoration $\mu _4$ is applicable to $|u_{k^*+1}|$ and minoration $\mu _1$ 
to $|u_{k^*}|$ and $|u_{k^*+2}|$.}
\end{rems} 

Consider the three numbers $u_1$, $u_2$, $u_3$. When represented in the form 
(\ref{eqform}), their exponents $-7/2+k$ are negative, so $|-zq^k|>1$. 
For fixed $\omega _k$, 
the modulus $|-zq^k|$ decreases as $|q|$ increases. 
Hence one can apply part (2) 
of Remarks~\ref{remsmu} and minorize $|u_k|$ by its value for $|q|=0.6$. 
Making use of Remarks~\ref{remsminoration} one finds that the 
product $|u_1|\cdot |u_2|\cdot |u_3|$ is minorized by the least of the 
$7$ numbers $\mu _1(0.6^{-5/2})\cdot \mu _4(0.6^{-3/2})\cdot \mu _1(0.6^{-1/2})$ and 
$\mu _{i_1}(0.6^{-5/2})\cdot \mu _{i_2}(0.6^{-3/2})\cdot \mu _{i_3}(0.6^{-1/2})$, where 
$(i_1,i_2,i_3)$ is a permutation of $(1,2,3)$. This is the number

$$\mu _3(0.6^{-5/2})\cdot \mu _2(0.6^{-3/2})\cdot \mu _1(0.6^{-1/2})=
1.742379963\ldots =:
\chi _0~.$$
Now consider for $j=1$, $2$, $\ldots$, 
a triple $|u_{3j+1}|$, $|u_{3j+2}|$, $|u_{3j+3}|$, $j\in \mathbb{N}$. Hence 
$|-zq^k|<1$ and by Remarks~\ref{remsmu} one has

$$\begin{array}{lll}
\tilde{u}_j:=|u_{3j+1}|\cdot |u_{3j+2}|\cdot |u_{3j+3}|\geq \min (A_j,B_j)&,&
{\rm where}\\ \\ 
A_j:=\mu _3(|q|^{-5/2+3j})\cdot \mu _2(|q|^{-3/2+3j})\cdot 
\mu _1(|q|^{-1/2+3j})&&
{\rm and}\\ \\ 
B_j:=\mu _1(|q|^{-5/2+3j})\cdot \mu _4(|q|^{-3/2+3j})\cdot 
\mu _1(|q|^{-1/2+3j})&.&
\end{array}$$
Set $\rho :=|q|$. We prove Lemma~\ref{lmk4} with the help of the following 
result (the proof is given at the end of this section):

\begin{lm}\label{lmAB}
The quantities $A_j(\rho )$ and $B_j(\rho )$ are decreasing in $\rho$ for 
$\rho \in (0,0.6]$.
\end{lm}
This means that one can minorize the product $\tilde{u}_j$ by 
$\chi _j:=\min (A_j(0.6),B_j(0.6))$. For $j=1$, $2$, $3$ and $4$, the values of 
$\chi _j$ are respectively
$$0.1749135662\ldots ~,~0.7772399345\ldots ~,~
0.9492771959\ldots ~,~0.9889171980\ldots $$
(in all cases they equal $A_j(0.6)$). For $k\geq 16$, the factors $|u_k|$ can 
be minorized by $|1-|q|^{-7/2+k}|$, and then by $|1-0.6^{-7/2+k}|$. We set 
$\chi _5:=\prod _{k=16}^{\infty}|1-0.6^{-7/2+k}|=0.9957913379\ldots$. Thus we 
minorize $|U|$ by 
$\chi _0\cdot \chi _1\cdot \chi _2\cdot \chi _3\cdot \chi _4\cdot \chi _5$.   

To minorize the product $R$ one can observe that for $|z|=|q|^{-7/2}$, each 
number $1+q^{j-1}/z$ is of the form (\ref{eqform}) with $k\geq 7$. Therefore 
when minorizing $|R|$ one can use the same reasoning as for $|U|$ and 
obtain a minoration by $\chi _2\cdot \chi _3\cdot \chi _4\cdot \chi _5$ 
(the first value of the index $k$ being $7$, not $1$, one has to skip 
the analogs of the minorations of $|u_k|$ for $k=1$, $\ldots$, $6$, i.e. 
to skip $\chi _0$ and $\chi _1$). Set 
$\chi _*:=\chi _2\cdot \chi _3\cdot \chi _4\cdot \chi _5$. Thus one can 
minorize the product $|Q|\cdot |U|\cdot |R|$ by 

$$1.2\cdot \chi _0\cdot \chi _1\cdot \chi _*^2=0.1930636291\ldots >
0.1851580824\ldots \geq |G|$$
from which the lemma follows.

\begin{proof}[Proof of Lemma~\ref{lmAB}:]
One has $A_j^2=CEF$, where 

$$C:=1+\rho ^{6j+6}~~~,~~~E:=1+\rho ^{6j+4}-\sqrt{2}\rho ^{3j+2}~~~,~~~
F:=1+\rho ^{6j+2}-2\rho ^{3j+1}~.$$
For $\rho \in (0,0.6]$, each factor $C$, $E$ and $F$ is positive-valued. 
Clearly $E'=(3j+2)\rho ^{3j+1}(2\rho ^{3j+2}-\sqrt{2})<0$, because 
$2\rho \leq 1.2<\sqrt{2}$. 
We show that $(CF)'<0$ from which and from $A_j>0$ follows that $A_j'<0$. 
A direct computation shows that 

$$\begin{array}{ccl}
(CF)'&=&\rho ^{3j}((6j+6)\rho ^{3j+5}+(6j+2)\rho ^{3j+1}+
(12j+8)\rho ^{9j+7}\\ \\ 
&&-(6j+2)-(18j+14)\rho ^{6j+6})~.\end{array}$$
For $j\geq 1$ and $\rho \in (0,0.6]$, the first three summands 
inside the brackets are majorized 
respectively by $(6j+6)\cdot 0.6^8$, $(6j+2)\cdot 0.6^4$ and 
$(12j+8)\cdot 0.6^{16}$. Their sum is less than $6j+2$, so $(CF)'<0$. 

We set $B_j^2=MNW$, where 

$$M:=(1-\rho ^{3j+1})^2~~~,~~~N:=1+\rho ^{6j+4}+\sqrt{2}\rho ^{3j+2}~~~,~~~
W:=(1-\rho ^{3j+3})^2~.$$
Obviously $W'<0$. We show that $(MN)'<0$ which together with $M>0$, 
$N>0$, $W>0$ and $B_j>0$ proves 
that $B_j'<0$. One has 
$(MN)'=(K_1+K_2-L_1)+(K_3+K_4-L_2)+(K_5-L_3)$, where
$$\begin{array}{lccclc}
K_1:=(6j+2)\rho ^{6j+1}&,&K_2:=\sqrt{2}(3j+2)\rho ^{3j+1}&,&
L_1:=(6j+2)\rho ^{3j}&,\\ \\ 
K_3:=(6j+4)\rho ^{6j+3}&,&K_4:=\sqrt{2}(9j+4)\rho ^{9j+3}&,&
L_2:=2\sqrt{2}(6j+3)\rho ^{6j+2}&,\\ \\ 
K_5:=(12j+6)\rho ^{12j+5}&&{\rm and}&&L_3:=(18j+10)\rho ^{9j+4}&.\end{array}$$
The inequality $K_5<L_3$ is evident. One has 
$K_4/L_2\leq (3/4)\rho ^{3j+1}\leq (3/4)\cdot 0.6^4=0.0972$ and 
$K_3/L_2\leq (10/9)\cdot (0.6/2\sqrt{2})=0.2357022603\ldots$, so 
$K_3+K_4<L_2$. Finally, $K_1/L_1\leq 0.6^{3j+1}\leq 0.6^4=0.1296$ and 
$K_2/L_1\leq \sqrt{2}\cdot 0.6=0.8485281372\ldots$ hence $K_1+K_2<L_1$ 
and $(MN)'<0$.
\end{proof}

\section{Proof of Lemma~\protect\ref{lmk1}
\protect\label{secprlmk1}}

We set 
$\theta ^{\dagger}(q,z):=\theta (q,z/q)=
\sum _{j=0}^{\infty}q^{j(j-1)/2}z^j=1+z+qz^2+q^3z^3+\cdots$. We show that 
$\theta ^{\dagger}(q,z)\neq 0$ for $|z|=|q|^{-1/2}$ from which the lemma follows. 
As $\theta ^{\dagger}(\bar{q},\bar{z})=\overline{\theta ^{\dagger}(q,z)}$, we 
consider only the case $\arg (q)\in [\pi /2,\pi]$. We assume that 
$|q|\in [c_0,0.6]$, because for $|q|\leq c_0$, the zeros of $\theta$ 
are strongly separated in modulus, see~\cite{Ko8}. Set

$$b_0:=1~~~,~~~b_1:=z~~~,~~~b_2:=qz^2~~~{\rm and}~~~
r_k:=\sum _{j=k}^{\infty}q^{j(j-1)/2}z^j~.$$
To show 
that $\theta ^{\dagger}\neq 0$ we prove that the modulus of the sum of some 
or all of the terms $b_0$, $b_1$ and $b_2$  
(we denote the set of the chosen terms by $S$) 
is larger than the modulus of the sum 
of the remaining terms of the series of $\theta ^{\dagger}$. 
We distinguish the following cases according to the intervals to which 
$\arg (z)$ and $\arg (q)$ belong:
\vspace{1mm} 

{\bf Case 1)} Re\,$z\geq 0$, i.e. $\arg (z)\in [-\pi /2,\pi /2]$, and 
$\arg (q)\in [\pi /2,\pi ]$. 
We set $S:= \{ b_0,b_1\}$. One has 

\begin{equation}\label{EQ1}
|r_2|\leq 
\sum _{j=2}^{\infty}|q|^{j(j-1)/2-j/2}=:\Phi _{\flat}(|q|)~~~\, ,~~~\, 
|r_3|\leq \sum _{j=3}^{\infty}|q|^{j(j-1)/2-j/2}=\Phi _{\flat}(|q|)-1
\end{equation}
and 

$$|1+z|\geq (1+|z|^2)^{1/2}=(1+|q|^{-1})^{1/2}=:\Phi _*(|q|)>\Phi _{\flat}(|q|)~.$$
The last inequality follows from $\Phi _*(|q|)$ and $\Phi _{\flat}(|q|)$ 
being respectively decreasing and increasing on $[c_0,0.6]$ and 

$$\Phi _*(0.6)=1.632993162\ldots >1.618354488\ldots =\Phi _{\flat}(0.6)~.$$

{\bf Case 2)} $\arg (z)\in [\pi ,3\pi /2]$ and 
$\arg (q)\in [\pi /2,\pi ]$. Then 
$\arg (qz)\in [3\pi /2,5\pi /2]$, i.e. Re\,$(qz)\geq 0$. 
We set $S:=\{ b_1,b_2 \}$. Hence 

$$|z+qz^2|=|z|\cdot |1+qz|\geq |q|^{-1/2}\cdot (1+|q|)^{1/2}=\Phi _*(|q|)~.$$
The other terms of $\theta ^{\dagger}$ 
are $1$, $q^3z^3$, $q^6z^4$, $\ldots$. For $|z|=|q|^{-1/2}$, 
their moduli equal respectively $1$, $|q|^{3/2}$, $|q|^4$, $\ldots$, 
which are precisely the terms of the series $\Phi _{\flat}$, so as in Case 1) 
one concludes that $|z+qz^2|\geq \Phi _*(|q|)>\Phi _{\flat}(|q|)\geq 1+|r_3|$. 
\vspace{1mm}

{\bf Case 3)} $\arg (z)\in [\pi /2,3\pi /4]$ and $\arg (q)\in [\pi /2,\pi ]$. 
The case is subdivided into five subcases in all of which we set 
$S:=\{ b_0,b_1,b_2\}$. 
\vspace{1mm}

{\bf Case 3A)} $\arg (z)\in [\pi /2,3\pi /4]$ and $\arg (q)\in [3\pi /4,\pi ]$. 
Then $\arg (qz^2)\in [7\pi /4,5\pi /2]$. If $\arg (qz^2)\in [2\pi ,5\pi /2]$, 
then 

$$\begin{array}{lllllllll}
|1+z+qz^2|&\geq&{\rm Im}\, (1+z+qz^2)&\geq& 
{\rm Im}\, (z)&\geq&\sin (3\pi /4)\cdot |q|^{-1/2}&\geq&\\ \\ 
(2\cdot 0.6)^{-1/2}&=&
0.9128709292\ldots &>&0.618354488\ldots &=&\Phi _{\flat}(0.6)-1&\geq&|r_3|~,
\end{array}$$ 
see (\ref{EQ1}). Set $\tau (|q|):=(2\cdot |q|)^{-1/2}-2^{-1/2}=
\sin (3\pi /4)\cdot |q|^{-1/2}+\sin (7\pi /4)=
-\cos (3\pi /4)\cdot |q|^{-1/2}-\cos (7\pi /4)$. 
If $\arg (qz^2)\in [7\pi /4,2\pi ]$, then 

$${\rm Im}\, (1+z+qz^2)\geq \tau (|q|)~~~{\rm and}~~~
{\rm Re}\, (1+z+qz^2)\geq 1-\tau (|q|)$$
(one can observe that $1-\tau (|q|)>0$ for $q\geq c_0$). Thus 

$$|1+z+qz^2|\geq (\tau ^2+(1-\tau )^2)^{1/2}=((1-2\tau )^2/2+1/2)^{1/2}~.$$
The latter quantity is minimal for $\tau =1/2$, i.e. for 
$|q|=0.3431457506\ldots$, when it equals 
$1/\sqrt{2}=0.7071067814\ldots >\Phi _{\flat}(0.6)-1\geq |r_3|$. 
\vspace{1mm}

{\bf Case 3B)} $\arg (z)\in [5\pi /8,3\pi /4]\subset [\pi /2,3\pi /4]$ and 
$\arg (q)\in [\pi /2,3\pi /4]$. 
Then $\arg (qz^2)\in [7\pi /4,9\pi /4]\subset [7\pi /4,5\pi /2]$ 
and one proves as in Case 3A) that $|1+z+qz^2|>|r_3|$.
\vspace{1mm}

{\bf Case 3C)} $\arg (z)\in [\pi /2,5\pi /8]$ and 
$\arg (q)\in [5\pi /8,3\pi /4]$. Then $\arg (qz^2)\in [13\pi /8,2\pi ]$. 
If $|q|\in [0.3,0.6]$, then  

$$\begin{array}{lllll}
|1+z+qz^2|&\geq& 
{\rm Re}\, (1+z+qz^2)&\geq&1+|q|^{-1/2}\cdot \cos (5\pi /8)+\cos (13\pi /8)\\ \\ 
&\geq&1+0.3^{-1/2}\cdot \cos (5\pi /8)+\cos (13\pi /8)&>&0.68~>~\Phi _{\flat}(0.6)-1
~\geq ~|r_3|~.\end{array}$$
If $|q|\in [c_0,0.3]$, then 

$$\begin{array}{llll}
&1+|q|^{-1/2}\cdot \cos (5\pi /8)+\cos (13\pi /8)&\geq& 
1+c_0^{-1/2}\cdot \cos (5\pi /8)+\cos (13\pi /8)\\ \\ \geq&0.5433422972\ldots &>&
\Phi _{\flat}(0.3)-1~=~0.1725370862\ldots ~\geq ~|r_3|~.\end{array}$$

{\bf Case 3D)} $\arg (z)\in [9\pi /16,5\pi /8]$ and 
$\arg (q)\in [\pi /2,5\pi /8]$. Then $\arg (qz^2)\in [13\pi /8,15\pi /8]$. 
Just as in Case 3C) one obtains $|1+z+qz^2|>|r_3|$. 
\vspace{1mm}

{\bf Case 3E)} $\arg (z)\in [\pi /2,9\pi /16]$ and 
$\arg (q)\in [\pi /2,5\pi /8]$. Then $\arg (qz^2)\in [3\pi /2,7\pi /4]$ and

$$\begin{array}{llllllll}
{\rm Re}\, (1+z+qz^2)&\geq &1+|q|^{-1/2}\cdot \cos (9\pi /16)&\geq& 
1+c_0^{-1/2}\cdot \cos (9\pi /16)&>&0.5712&,\\ \\ 
{\rm Im}\, (1+z+qz^2)&\geq &|q|^{-1/2}\cdot \sin (9\pi /16)-1&\geq&
0.6^{-1/2}\cdot \sin (9\pi /16)-1&>&0.2661&,\end{array}$$
so $|1+z+qz^2|>(0.5712^2+0.2661^2)^{1/2}>0.63>\Phi _{\flat}(0.6)-1\geq |r_3|$.
\vspace{1mm}

{\bf Case 4)} $\arg (z)\in [3\pi /4,\pi ]$ and 
$\arg (q)\in [\pi /2,\pi ]$. Then $\arg (qz^2)\in [2\pi ,3\pi ]$, i.e. 
Im\,$(qz^2)\geq 0$. We consider four subcases in all of which we set 
$S:=\{ b_0,b_1,b_2\}$:
\vspace{1mm}

{\bf Case 4A)} $\arg (z)\in [3\pi /4,5\pi /6]$ and 
$\arg (q)\in [\pi /2,\pi ]$. In this case 
Im\, $(z)\geq 0.6^{-1/2}\cdot \sin (5\pi /6)$ and 

$$|1+z+qz^2|\geq {\rm Im}\, 
(1+z+qz^2)\geq 0.6^{-1/2}\cdot \sin (5\pi /6)=0.64\ldots 
>\Phi _{\flat}(0.6)-1\geq |r_3|~.$$

{\bf Case 4B)} $\arg (z)\in [5\pi /6,11\pi /12]$ and 
$\arg (q)\in [\pi /2,\pi ]$. Hence $\arg (qz^2)\in [13\pi /6,17\pi /6]$ 
(with $\sin (13\pi /6)=\sin (17\pi /6)=1/2$) and 

$$|1+z+qz^2|\geq {\rm Im}\, (1+z+qz^2)\geq 0.6^{-1/2}\cdot \sin (11\pi /12)+
\sin (13\pi /6)=0.83\ldots >\Phi _{\flat}(0.6)-1\geq |r_3|~.$$

{\bf Case 4C)} $\arg (z)\in [11\pi /12,\pi ]$ and 
$\arg (q)\in [\pi /2,5\pi /6]$ hence $\arg (qz^2)\in [7\pi /3,17\pi /6]$. 
If $\arg (qz^2)\in [7\pi /3,5\pi /2]$, then 

$$|1+z+qz^2|\geq {\rm Im}\, (1+z+qz^2)\geq {\rm Im}\, (qz^2)\geq \sin (7\pi /3)
=\sqrt{3}/2=
0.86\ldots >\Phi _{\flat}(0.6)-1\geq |r_3|~.$$
If $\phi :=\arg (qz^2)\in [5\pi /2,17\pi /6]$, then 

$$\begin{array}{llllll}
{\rm Im}\, (1+z+qz^2)&\geq&{\rm Im}\, (qz^2)&=&\sin (\phi )&,\\ \\ 
|{\rm Re}\, (1+z+qz^2)|&&&\geq&|0.6^{-1/2}\cdot \cos (11\pi /12)+1+\cos (\phi )|&
{\rm and}\\ \\ 
|1+z+qz^2|&\geq&T(\phi )&:=&(\sin ^2(\phi )+
(0.6^{-1/2}\cdot \cos (11\pi /12)+1+\cos (\phi ))^2)^{1/2}&.\end{array}$$
The function $T(\phi )$ takes only values larger than $1$ (hence larger than 
$|r_3|$) for $\phi \in [5\pi /2,17\pi /6]$. 
\vspace{1mm}

{\bf Case 4D)} $\arg (z)\in [11\pi /12,\pi ]$ and 
$\arg (q)\in [5\pi /6,\pi ]$, so $\arg (qz^2)\in [8\pi /3,3\pi ]$ and 
$${\rm Re}\, (1+z+qz^2)\leq 0.6^{-1/2}\cdot \cos (11\pi /12)+1+\cos (8\pi /3)
=-0.7470048804\ldots ~.$$
Hence $|1+z+qz^2|\geq |{\rm Re}\, (1+z+qz^2)|>\Phi _{\flat}(0.6)-1\geq |r_3|$.

\section{Proof of Lemma~\protect\ref{lmk2}
\protect\label{secprlmk2}}

We set $z:=q\xi$ hence $|\xi |=|q|^{-3/2}$. 
Thus $\theta =\sum _{j=0}^{\infty}q^{j(j-1)/2}\xi ^j$. 
Next we set $A:=1+\sum _{j=4}^{\infty}q^{j(j-1)/2}\xi ^j$ and $B:=1+q\xi +q^3\xi ^2$, 
so $\theta =A+\xi B$. 
Finally we set $\xi :=q\zeta$, thus $B=1+\zeta +q\zeta ^2$ and 
$|\zeta |=|q|^{-1/2}$. 

\begin{lm}\label{lmABq}
(1) For $|q|\leq 0.6$ and $|\xi |\leq |q|^{-3/2}$, one has 
$|A|\leq a_0:=1+\sum _{j=4}^{\infty}0.6^{j(j-1)/2-3j/2}=2.330487021\ldots$.

(2) For $|q|\leq 0.6$, $\arg (q)\in [2\pi /3,\pi ]$ and 
$|\xi |\leq |q|^{-3/2}$, one has 
$|\xi B|>a_0\geq |A|$ hence $\theta \neq 0$. 

(3) For $|q|\leq 0.55$ and $|\xi |\leq |q|^{-3/2}$, one has 
$|\xi B|>a_0\geq |A|$ hence $\theta \neq 0$. 
\end{lm}

\begin{rem}\label{remnum}
{\rm For $\arg (q)\in [\pi /2,2\pi /3]$, it is possible to show numerically 
that $|\xi B|>a_0\geq |A|$ also for $|q|\in [0.55,0.6]$. To this end one can 
compute the quantities $|\xi B|$ and $|A|$ with sufficiently small steps in 
$\arg (q)$ and $|q|$. To obtain a majoration of $|A|$ one can set 
$A:=A^*+A^{**}$ with $A^*:=1+\sum _{j=4}^{7}q^{j(j-1)/2}\xi ^j$ and 
$A^{**}:=\sum _{j=8}^{\infty}q^{j(j-1)/2}\xi ^j$. For $|\xi |=|q|^{-3/2}$ and 
$|q|\leq 0.6$, one has $|A^{**}|\leq \sum _{j=8}^{\infty}0.6^{j(j-1)/2-3j/2}=
0.0002925303367\ldots$. Thus there remains to estimate $|\xi B|$ and $|A^*|$ 
which contain finitely-many terms.}
\end{rem}

\begin{proof}[Proof of Lemma~\ref{lmABq}]
Part (1) follows from $|A|\leq 1+\sum _{j=4}^{\infty}|q|^{j(j-1)/2-3j/2}\leq a_0$. 
To prove parts (2) and (3) we set $q:=\rho e^{i\omega}$, $\rho \geq 0$, 
$\omega \in [\pi /2,\pi ]$, and 
$\zeta :=\rho ^{-1/2}e^{i\psi}$, $\psi \in [0,2\pi ]$. 
Observe that $|q\zeta ^2|=1$. Thus 

$$\begin{array}{ccl}
B&=&1+\rho ^{-1/2}e^{i\psi}+e^{i(2\psi +\omega )}~~~\, {\rm and}\\ \\ 
|B|^2&=&(1+\cos (2\psi +\omega )+\rho ^{-1/2}\cos \psi )^2+
(\sin (2\psi +\omega )+\rho ^{-1/2}\sin \psi )^2\\ \\ 
&=&2+\rho ^{-1}+2\rho ^{-1/2}\cos (\psi +\omega )+
2\cos (2\psi +\omega )+2\rho ^{-1/2}\cos \psi \\ \\ 
&=&2+\rho ^{-1}+4\rho ^{-1/2}\cos (\psi +\omega /2)\cos (\omega /2)+
2\cos (2\psi +\omega )\\ \\ 
&=&\rho ^{-1}+4\rho ^{-1/2}\cos (\psi +\omega /2)\cos (\omega /2)+
4\cos ^2(\psi +\omega /2)~.\end{array}$$
Set $a:=\cos (\psi +\omega /2)\in [-1,1]$ and 
$b:=\cos (\omega /2)\in [0,\sqrt{2}/2]$ (because 
$\omega /2\in [\pi /4,\pi /2]$). The polynomial $4a^2+4\rho ^{-1/2}ab+\rho ^{-1}$ 
(considered as a polynomial in $a$ with $b$ as a parameter) 
takes its minimal value for $a=-\rho ^{-1/2}b/2$. This value is 
$(1-b^2)\rho ^{-1}$. Therefore for $b\leq 1/2$ 
(hence for $\omega \in [2\pi /3,\pi ]$), one has 
$|B|^2\geq 3\cdot \rho ^{-1}/4\geq 3\cdot 0.6^{-1}/4$. So for 
$|\xi |=|q|^{-3/2}\geq 0.6^{-3/2}$, one has 

$$|\xi B|=|q|^{-3/2}\cdot |B|\geq 0.6^{-3/2}\cdot (3\cdot 0.6^{-1}/4)^{1/2}=
2.405626123\ldots >a_0\geq |A|~.$$
For $\rho =|q|\leq 0.55$, one has $|B|^2\geq \rho ^{-1}/2$ and 
$|\xi B|\geq 0.55^{-2}/\sqrt{2}=2.337543079\ldots >a_0\geq |A|$. 
\end{proof}

\end{document}